\documentclass{amsart}
\usepackage{amsmath}
\usepackage{amsxtra}
\usepackage{amscd}
\usepackage{amsthm}
\usepackage{amsfonts}
\usepackage{amssymb}
\usepackage{mathdots}
\usepackage{bbm}
\usepackage{graphicx}
\usepackage{cite}
\usepackage{tikz-cd}
\usepackage[bookmarksnumbered, colorlinks, plainpages]{hyperref}
\usepackage[utf8]{inputenc}
\usepackage{color}
\usepackage{enumerate}
\newtheorem{theorem}{Theorem}[section]

\newtheorem{lemma}[theorem]{Lemma}

\newtheorem{corollary}[theorem]{Corollary}

\theoremstyle{definition}

\usepackage[paperwidth=17.10cm, paperheight=24.60cm, left=1.5cm, right=1.5cm, top=2cm, bottom=2cm]{geometry}

\usepackage{etoolbox}
\patchcmd{\section}{\scshape}{\large}{}{}
\patchcmd{\subsection}{\bfseries}{\normalfont}{}{}
\patchcmd{\subsubsection}{\itshape}{\normalfont}{}{}

\makeatletter
\def\@setauthors{%
	\begingroup
	\def\thanks{\protect\thanks@warning}%
	\trivlist
	\centering\footnotesize \@topsep30\p@\relax
	\advance\@topsep by -\baselineskip
	\item\relax\normalsize 
	\author@andify\authors
	\def\\{\protect\linebreak}%
	\authors%
	\ifx\@empty\contribs
	\else
	,\penalty-3 \space \@setcontribs
	\@closetoccontribs
	\fi
	\endtrivlist
	\endgroup
}
\def\@settitle{\begin{center}%
		\baselineskip14\p@\relax
		\bfseries\large
		\@title
	\end{center}%
}
\makeatother

\usepackage{titlesec}
\titleformat{\section}
{\normalfont\large\bfseries}
{\thesection}{1em}{}
\titleformat{\subsection}
{\normalfont\large\bfseries}
{\thesubsection}{1em}{}

\usepackage{tikz,xcolor,hyperref}

\definecolor{lime}{HTML}{A6CE39}
\DeclareRobustCommand{\orcidicon}{%
	\begin{tikzpicture}
		\draw[lime, fill=lime] (0,0) 
		circle [radius=0.16] 
		node[white] {{\fontfamily{qag}\selectfont \tiny ID}};
		\draw[white, fill=white] (-0.0625,0.095) 
		circle [radius=0.007];
	\end{tikzpicture}
	\hspace{-2mm}
}

\newcommand{\orcidHa}{\href{https://orcid.org/0000-0002-9129-8441}{\orcidicon}}
\newcommand{\orcidSon}{\href{https://orcid.org/0000-0002-9560-6392}{\orcidicon}}
\newcommand{\orcidVinh}{\href{https://orcid.org/0009-0005-9917-1436}{\orcidicon}}

\begin{document}
	\title[A note on additive commutator groups in certain algebras]{A note on additive commutator groups in certain algebras}
	
	\author[Nguyen Thi Thai Ha]{Nguyen Thi Thai Ha\textsuperscript{1,$\S$}}
	
	\author[Tran Nam Son]{Tran Nam Son\textsuperscript{2,3,\ddag,*}}
	
	\author[Pham Duy Vinh]{Pham Duy Vinh\textsuperscript{2,3,4,\dag}}
	
	
	
	\maketitle	
	
	\begin{center}
	{\small 
		*Corresponding author\\
		\textsuperscript{1}Campus in Ho Chi Minh City,\\ University of Transport and Communications,\\
		Ho Chi Minh City, Vietnam\\
		\textsuperscript{2}Faculty of Mathematics and Computer Science,\\ University of Science,\\ Ho Chi Minh City, Vietnam \\
		\textsuperscript{3}Vietnam National University,\\ Ho Chi Minh City, Vietnam\\
		\textsuperscript{4}Department of Mathematics,\\ Dong Nai University,\\ 9 Le Quy Don Str., Tan Hiep Ward, Bien Hoa City,\\ Dong Nai Province, Vietnam\\
		\textsuperscript{$\S$}hantt\_ph@utc.edu.vn,\\
		\orcidHa\url{http://orcid.org/0000-0002-9129-8441},\\
		\textsuperscript{\ddag}trannamson1999@gmail.com,\\ \orcidSon \url{https://orcid.org/0000-0002-9560-6392},\\
		\textsuperscript{\dag}phamduyvinh.dlu@gmail.com,\\ \orcidVinh\url{https://orcid.org/0009-0005-9917-1436}
	}
\end{center}
	
	\begin{abstract}
We study whether a unital associative algebra \( A \) over a field  admits a decomposition of the form $A = Z(A) + [A,A]$ where \( Z(A) \) is the center of \( A \) and \( [A,A] \) denotes the additive subgroup of $A$ generated by all additive commutators of $A$. Among our main considerations are the cases in which $A$ is the matrix ring over a division ring, a generalized quaternion algebra, or a semisimple finite-dimensional algebra. We also discuss some applications  that do not necessarily require the decomposition, such as the case where \( A \) is the twisted group algebra of a locally finite group over a field of characteristic zero: if all additive commutators of $A$ are central, then \( A \) must be commutative.
	\end{abstract}
	
	\section{Introduction}
	
	In any associative algebra, one way to measure how far two elements are from commuting is to consider their additive commutator, defined by $ab-ba$. When the commutator is zero, the two elements commute; otherwise, the expression captures the essence of their noncommutativity. This simple operation plays a central role in algebra; for instance, it allows us to view associative algebras as Lie algebras by treating the commutator as a Lie bracket. From a structural perspective, it is natural to begin with a familiar example: the ring $\mathrm{M}_n(F)$ of $n\times n$ matrices over a field $F$ of of characteristic zero. In this setting, any matrix $A$ in $\mathrm{M}_n(F)$ can be  written as the sum of two parts: a central matrix $\frac{\mathrm{trace}(A)}{n}\mathrm{I}_n$ and a traceless matrix $A-\frac{\mathrm{trace}(A)}{n}\mathrm{I}_n$. What makes this decomposition particularly striking is a result by Albert and Muckenhoupt \cite{Pa_Al_57}, which shows that every traceless matrix is in fact a single additive commutator. In other words, each matrix in $\mathrm{M}_n(F)$ can be expressed as the sum of a central element and an additive commutator. Such a decomposition naturally prompts a broader question: does this phenomenon extend beyond $\mathrm{M}_n(F)$? That is, given a unital associative algebra $A$ over a field, can every element of $A$ be expressed as the sum of a central element of $A$ and an additive commutator of $A$? A more modest question one might ask is whether the equality $A = Z(A) + [A,A]$ holds, where \( Z(A) \) is the center of \( A \) and \( [A,A] \) denotes the additive subgroup of $A$ generated by all additive commutators of $A$. In this way, the decomposition separates the “commutative core” of the algebra from the part that captures its internal noncommutativity.
	
The paper will open with a structural insight into matrix rings over division rings, stated precisely in Theorem~\ref{key1}. This leads naturally to several corollaries in broader contexts, including generalized quaternion algebras (Corollary~\ref{quaternion}), semisimple finite-dimensional algebras (Corollary~\ref{semisimple}), and C*-algebras (Corollary~\ref{C-algebra}). We also derive a version of the classical commutativity theorem for twisted group algebras (Corollary~\ref{twisted}). Moreover, we prove a similar additive decomposition in Theorem~\ref{key2} for images of noncommutative polynomials. Finally, in connection with this result, we show that under suitable conditions, any division ring that is finite-dimensional over its center is generated by the images of noncommutative polynomials (as an algebra) (see Theorem~\ref{generated}).

\section{The main results}

We adopt the standard notation \( \mathrm{M}_n(A) \) to refer to the algebra of \( n \times n \) matrices over a unital associative algebra \( A \) over a field. As established in \cite[Lemma 5.8]{Pa_DuHaSo_24}, if \( D \) is an algebraic division ring of characteristic zero and \( n \) is a positive integer, then the matrix algebra \( \mathrm{M}_n(D) \) admits the decomposition  
$\mathrm{M}_n(D) = Z(\mathrm{M}_n(D)) + [\mathrm{M}_n(D), \mathrm{M}_n(D)].$  It is well known that the center \( Z(\mathrm{M}_n(D)) \) consists precisely of all scalar matrices of the form \( \lambda \mathrm{I}_n \), where \( \lambda \) belongs to the center of \( D \). Recall that a division ring \( D \) is called \emph{algebraic} if every element of \( D \) satisfies a nonzero polynomial with coefficients in the center of \( D \). In the setting of positive characteristic, further insight is provided by \cite[Theorem 5]{Pa_ChuLee_79}, which focuses on division rings that are finite-dimensional over their centers, a special subclass of algebraic division rings. According to \cite[(15.8) Theorem, p.~242]{Bo_La_91}, such division rings must have dimensions that are perfect squares. This naturally suggests considering the square root of the dimension when formulating the condition in the next result.

\begin{theorem} \label{key1}
	Let \( D \) be a division ring with center \( F \), and let \( n \) be a positive integer. Then, the equality
	$\mathrm{M}_n(D) = Z(\mathrm{M}_n(D)) + [\mathrm{M}_n(D), \mathrm{M}_n(D)]$
	holds in either of the following cases:
	\begin{enumerate}[\rm (i)]
		\item \( D \) is an algebraic division ring of characteristic zero.
		\item \( D \) is finite-dimensional over \( F \), where \( \operatorname{char} (F) = p > 0 \), and \( p \) does not divide \( \sqrt{\dim_F D} \) nor $n$.
	\end{enumerate}
\end{theorem}

\begin{proof}
	Part (i) follows directly from \cite[Lemma 5.8]{Pa_DuHaSo_24}. It remains to consider case (ii). Since \( p \) does not divide \( \sqrt{\dim_F D} \), it follows from \cite[Theorem 5]{Pa_ChuLee_79} that  
	$D = F + [D, D].$  
	By combining this claim with \cite[Theorems 1 and 3]{Pa_ChuLee_79}, we obtain  
	$\mathrm{M}_n(D) = \mathrm{M}_n(F) + [\mathrm{M}_n(D), \mathrm{M}_n(D)].$  	Thus, any matrix \( A \in \mathrm{M}_n(D) \) can be expressed as the sum of some \( B \in \mathrm{M}_n(F) \) and some \( C \in [\mathrm{M}_n(D), \mathrm{M}_n(D)] \). Since \( p \) does not divide \( n \), it follows that the matrix $A$ can be written as:  
	$A = \left( B - \frac{1}{n} \mathrm{trace}(B) \mathrm{I}_n \right) + \frac{1}{n} \mathrm{trace}(B) \mathrm{I}_n + C.$  
	Moreover, since the matrix $B - \frac{1}{n} \mathrm{trace}(B)$ is traceless, it follows from \cite[Lemma~2]{Pa_ChuLee_79} that $ B - \frac{1}{n} \mathrm{trace}(B) \mathrm{I}_n \in [\mathrm{M}_n(D), \mathrm{M}_n(D)],$ which implies that  
	$A \in Z(\mathrm{M}_n(D)) + [\mathrm{M}_n(D), \mathrm{M}_n(D)].$
	Consequently, we obtain the following:  
	$\mathrm{M}_n(D) = Z(\mathrm{M}_n(D)) + [\mathrm{M}_n(D), \mathrm{M}_n(D)].$
	This completes the proof.
\end{proof}

In situations where the assumptions of Theorem~\ref{key1} are not met, the corresponding equality does not necessarily hold. For example, consider the base case \( n = 1 \). As noted in the comment following \cite[Lemma 3]{Pa_As_66}, if \( \operatorname{char}(F) = p > 0 \) and \( \sqrt{\dim_F D} \) is divisible by \( p \), then it follows that \( F \subseteq [D, D] \). Meanwhile, \cite{Pa_Me_06} shows that if \( \dim_F D < \infty \), then \( D \ne [D, D] \). Therefore, under the same assumptions (positive characteristic and \( \sqrt{\dim_F D} \) divisible by \( p \)), we conclude that \( D \ne F + [D, D] \). This observation highlights the importance of the base case \( n = 1 \) in the proof of Theorem~\ref{key1}; if this case fails, the argument for \( n > 1 \) cannot proceed as intended. 

For division rings that are not algebraic over their centers, potential counterexamples may be found in \cite{Pa_Ha_19} (in one variable) or in \cite[Example 4]{Pa_Ha_02} (in several variables), where constructions inspired by Hilbert suggest that the center \( F \) of $D$ is very small and may fail to capture much of \( D \). However, we do not claim to have verified such a counterexample.

Keep Theorem~\ref{key1} in mind, we now turn to some of its applications. We begin with generalized quaternion algebras, as they form a class of algebraic structures closely related to those considered in Theorem~\ref{key1}.

A \textit{generalized quaternion algebra} over a field \( F \) is a ring containing \( F \) that is also a $4$-dimensional vector space over \( F \), with basis \( \{1, i, j, k\} \). The multiplication is defined according to the following rules:
\begin{enumerate}[\rm (i)]
	\item If \( \operatorname{char}(F) \neq 2 \), then \( i^2, j^2 \in F \setminus \{0\} \) and \( k = ij = -ji \).
	\item If \( \operatorname{char}(F) = 2 \), then \( i^2 + i \in F \), \( j^2 \in F \setminus \{0\} \), and \( k = ij = j(i + 1) \).
\end{enumerate}
According to \cite[Main Theorem~5.4.4 and Theorem~6.4.11]{Bo_Vo_21}, every quaternion algebra over \( F \) is either a division ring or isomorphic to \( \mathrm{M}_2(F) \) as an \( F \)-algebra. With this clarification in mind, we will now use it to establish the following corollary.

\begin{corollary}\label{quaternion}
	If $A$ is a generalized quaternion algebra over a field $F$ of characteristic $0$, then $\mathrm{M}_n(A) = Z(\mathrm{M}_n(A)) + [\mathrm{M}_n(A), \mathrm{M}_n(A)]$ for all positive integers $n$.
\end{corollary}

\begin{proof}
According to \cite[Main Theorem~5.4.4 and Theorem~6.4.11]{Bo_Vo_21}, we know that \( A \) is either a division ring or isomorphic to \( \mathrm{M}_2(F) \) as an \( F \)-algebra. The case where \( A \) is a division ring is already addressed by Theorem~\ref{key1}. The remaining case is when \( A \) is isomorphic to \( \mathrm{M}_2(F) \) as an \( F \)-algebra. In this scenario, for any positive integer $n$, it follows that \( \mathrm{M}_n(A) \) is isomorphic to \( \mathrm{M}_{2n}(F) \) as an \( F \)-algebra, and thus, the result follows from Theorem~\ref{key1}.
\end{proof}

We now proceed with the discussion of semisimple algebras. A \textit{semisimple algebra} over a field \( F \) is an associative Artinian \( F \)-algebra with zero Jacobson radical. According to the \textit{Wedderburn--Artin theorem}, when the algebra is finite-dimensional over \( F \), this is equivalent to it being isomorphic to a finite direct product of matrix algebras over division rings that are themselves finite-dimensional over \( F \) (see \cite[Main Theorem~7.3.10]{Bo_Vo_21}). With these foundational facts in place, we can now derive the following corollaries with ease.

\begin{corollary}\label{semisimple}
If \( A \) is a semisimple finite-dimensional algebra over a field \( F \) of characteristic zero, then $A=Z(A)+[A,A]$.
\end{corollary}

\begin{proof}
	We begin with the base case \( n = 1 \). Since \( A \) is a semisimple finite-dimensional $F$-algebra, the Wedderburn--Artin theorem implies that \( A \) is isomorphic to a finite direct product of matrix algebras over division rings, each of which is finite-dimensional over its center. Thus, any element \( \alpha \in A \) can be written, without loss of generality, as \( \alpha = (\alpha_1, \alpha_2, \ldots, \alpha_t) \), where each \( \alpha_i \in \mathrm{M}_{n_i}(D_i) \) for some division ring \( D_i \) with center \( F_i \) and some positive integer $n_i$, and \( \dim_{F_i} D_i < \infty \). By Theorem~\ref{key2}, each matrix algebra \( \mathrm{M}_{n_i}(D_i) \) admits the decomposition: $\mathrm{M}_{n_i}(D_i) = Z(\mathrm{M}_{n_i}(D_i)) + [\mathrm{M}_{n_i}(D_i), \mathrm{M}_{n_i}(D_i)].$	It follows that each \( \alpha_i \) can be expressed as \( \alpha_i = \beta_i + \gamma_i \), where \( \beta_i \in Z(\mathrm{M}_{n_i}(D_i)) \) and \( \gamma_i \in [\mathrm{M}_{n_i}(D_i), \mathrm{M}_{n_i}(D_i)] \). Consequently, we have:
	$\alpha = (\beta_1, \beta_2, \ldots, \beta_t) + (\gamma_1, \gamma_2, \ldots, \gamma_t),$
	where \( (\beta_1, \ldots, \beta_t) \in Z(A) \) and \( (\gamma_1, \ldots, \gamma_t) \in [A, A] \), which completes the argument for the case \( n = 1 \). The case \( n > 1 \) proceeds in a similar fashion to the proof of Theorem~\ref{key1} and is therefore omitted.
\end{proof}

A particularly notable class of examples fitting the framework of Corollary~\ref{semisimple} is that of C*-algebras. As shown in \cite[Corollary I.9.13]{Bo_Da_96}, every C*-algebra has a trivial Jacobson radical. Since any finite-dimensional algebra is Artinian, it follows that every finite-dimensional C*-algebra is semisimple. This immediately yields the following corollary in the finite-dimensional case. In this setting, such algebras are, up to unitary isomorphism as \( \mathbb{C} \)-algebras, precisely the finite direct products of matrix algebras over \( \mathbb{C} \).

\begin{corollary}\label{C-algebra}
If \( A \) is a finite-dimensional C*-algebra, then $A=Z(A)+[A,A]$.
\end{corollary}

In addition to the C*-algebras discussed in Corollary~\ref{C-algebra}, similar results can be found in \cite{Pa_Thi_24}, though we will briefly summarize them here for completeness. In particular, \cite[Proposition 5.2]{Pa_Thi_24} states that if \( A \) is a unital C*-algebra, then \( A = Z(A) + [A, A] \) in the following cases:
\begin{enumerate}[\rm (1)]
	\item If \( A \) is pure, every quasitrace on \( A \) is a trace, and \( A \) has the Dixmier property.
	\item If \( A \) has no tracial states (in particular, if \( A \) is properly infinite).
	\item If \( A \) is a von Neumann algebra.
\end{enumerate} Recall that a unital C*-algebra \( A \) is said to have the Dixmier property if, for every element \( a \in A \), the closed convex hull of the unitary orbit \( \{ u a u^* : u \in A \text{ is unitary} \} \) intersects non-trivially with the center \( Z(A) \). It is shown in \cite{Pa_Di_49} that every von Neumann algebra has the Dixmier property, while an example of a unital C*-algebra without this property is also provided. Later, in \cite[Theorem 1.1]{Pa_Ar_17}, it is established that the Dixmier property holds if and only if \( A \) is weakly central, every simple quotient of \( A \) has at most one tracial state, and every extreme tracial state factors through a simple quotient. 

For the proof of (1), if \( A \) has the Dixmier property, then every element in \( A \) is a Dixmier element, as defined in \cite{Pa_Ar_23}. By \cite[Lemma 4.1]{Pa_Ar_23}, we conclude that \( A = Z(A) + \overline{[A, A]} \). Moreover, if \( A \) is pure and every quasitrace is a trace, then \( [A, A] \) is closed by \cite[Theorem 1.1]{Pa_Ng_16}. Combining these results, we obtain \( A = Z(A) + [A, A] \). For proof of (2), \cite[Theorem~1]{Pa_Pop_02} shows that if \( A \) has no tracial states, then \( A = [A, A] \). For proof of (3), it was shown in \cite[Theorem 5.19]{Pa_Bre_08} that every von Neumann algebra \( A \) satisfies \( A = Z(A) + [A, A] \).

In addition, \cite[Theorem 2.3]{Pa_Le_23} provides some necessary conditions for an algebra \( A \) to admit the decomposition \( A = Z(A) + [A, A] \). Rather than delving into these conditions here, we shift our focus to a version of the commutativity theorem in the setting of twisted group algebras, as presented below.

For the reader’s convenience, we begin with a brief overview. Let \( F \) be a field and let \( G \) be a multiplicative group. The twisted group algebra \( F^\tau G \) is a unital associative \( F \)-algebra with basis \( \overline{G} = \{ \overline{g} \mid g \in G \} \), a formal copy of \( G \) over \( F \). Addition is defined in the usual way, while multiplication is given by \( \overline{x} \, \overline{y} = \tau(x, y)\,\overline{xy} \) for all \( x, y \in G \), where the twisting function \( \tau : G \times G \to F \setminus \{0\} \) encodes the deformation.

As shown in \cite[Theorem 4.2]{Bo_Pa_89}, if \( G \) is finite and the characteristic of \( F \) does not divide the order of \( G \), then the Jacobson radical of \( F^\tau G \) is trivial. In such cases, \( F^\tau G \) is also finite-dimensional over \( F \).

To explore infinite-dimensional settings, we consider locally finite groups instead of finite ones. Recall that a group is called \textit{locally finite} if every finitely generated subgroup is finite. Clearly, all finite groups are locally finite, but the class also includes infinite examples, such as the Prüfer groups. Choosing an infinite locally finite group \( G \), the resulting twisted group algebra \( F^\tau G \) becomes infinite-dimensional over \( F \). However, in this broader context, we can no longer rely on group order to ensure the triviality of the Jacobson radical. To apply the conclusion of \cite[Theorem 4.2]{Bo_Pa_89} in the infinite case, we therefore restrict attention to fields \( F \) of characteristic zero.

With these considerations in place, we now state the following corollary.

\begin{corollary}\label{twisted}
	Let \( F^\tau G \) be a twisted group algebra of a locally finite group \( G \) over a field \( F \) of characteristic $0$, associated with the twisting function \( \tau \). If all additive commutators of $F^\tau G$ belong to the center $Z(F^\tau G)$ of $F^\tau G$, then $F^\tau G$ must be commutative.
\end{corollary}

\begin{proof}
Let \( \alpha, \beta \) be arbitrary elements of \( F^\tau G \). We begin by expressing them as finite linear combinations:
$\alpha = \sum_{g \in G} \lambda_g \overline{g}, \quad \beta = \sum_{h \in G} \gamma_h \overline{h},$
where \( \lambda_g, \gamma_h \in F \), and all but finitely many coefficients vanish. Define the supports of \( \alpha \) and \( \beta \) as
$\mathrm{supp}(\alpha) = \{ g \in G \mid \lambda_g \neq 0 \},$ and $\mathrm{supp}(\beta) = \{ h \in G \mid \gamma_h \neq 0 \}.$ Clearly, both $ \mathrm{supp}(\alpha)$ and  $\mathrm{supp}(\beta) $ are finite.
Let \( H = \langle \mathrm{supp}(\alpha), \mathrm{supp}(\beta) \rangle \) be the subgroup of \( G \) generated by these supports. Since \( \alpha \) and \( \beta \) are linear combinations of basis elements from \( H \), it is clear that \( \alpha, \beta \in F^\tau H \). Because \( G \) is locally finite, the subgroup \( H \) is necessarily finite. Hence, \( F^\tau H \) is a finite-dimensional twisted group algebra over \( F \). By \cite[Theorem 4.2]{Bo_Pa_89} and Corollary~\ref{semisimple}, we know that
$F^\tau H = Z(F^\tau H) + [F^\tau H, F^\tau H].$
Under the assumption that all additive commutators in \( F^\tau G \) lie in its center \( Z(F^\tau G) \), it follows that
$[F^\tau H, F^\tau H] \subseteq Z(F^\tau G) \subseteq Z(F^\tau H).$
Therefore, \( F^\tau H = Z(F^\tau H) \), implying that \( F^\tau H \) is commutative. Since \( \alpha, \beta \in F^\tau H \), it follows that \( \alpha \beta = \beta \alpha \), completing the proof.
\end{proof}

To the best of our knowledge, the original form of Corollary~\ref{twisted} was first established in the context of division rings. In particular, \cite[Lemma, p.~148]{Pa_An_47} contains a result by Ancochea stating that if an element \( c \) in a division ring \( D \) commutes with all additive commutators in \( D \), then \( c \in Z(D) \). This was later revisited by Herstein, who provided an elementary proof in \cite[Theorem~2]{Pa_Her_53}. He also observed that the assumption that \( D \) is a division ring is not necessary; the conclusion remains valid if \( D \) contains no nonzero nilpotent elements. As a consequence, in any such ring \( D \), if all additive commutators lie in the center, then \( D \) must be commutative.

 It is worth noting that the additive subgroup \( [A, A] \) of an algebra \( A \), generated by all commutators, is equal to the linear span of those commutators. This also motivates our focus on the linear spans of images of noncommutative polynomials. Hence, we continue with expressing our interest in exploring noncommutative polynomials beyond the commutators. To set the stage, we recall a useful result from \cite[Corollary 2.6]{Pa_Bre_20}, which we present here as a lemma for convenient reference, and some preliminaries.

\begin{lemma} {\rm \cite[Corollary 2.6]{Pa_Bre_20}} \label{key}
	Let \( A \) be a simple algebra over an infinite field \( F \). Suppose \( f \in F\langle\mathcal{X}\rangle \) is a noncommutative polynomial that is neither a polynomial identity nor a central polynomial of \( A \). Then, the commutator subspace \( [A, A] \) is contained in the linear span of \( f(A) \); that is, $[A, A] \subseteq \mathrm{span}f(A).$
\end{lemma}

Here, \( F\langle\mathcal{X}\rangle \) denotes the free $F$-algebra generated by the set \( \mathcal{X} = \{x_1, x_2, \ldots\} \), that is, the unital associative \( F \)-algebra of noncommutative polynomials in the variables \( x_i \). For any unital associative \( F \)-algebra \( A \), and any  polynomial \( f = f(x_1, \ldots, x_m) \in F\langle\mathcal{X}\rangle \) (for some positive integer \( m \)), define the \textit{image} of \( f \) on \( A \) by
$f(A) = \{f(a_1, \ldots, a_m) \mid a_1, \ldots, a_m \in A\}.$
We say that \( f\in F\langle\mathcal{X}\rangle\setminus\{0\} \) is a \textit{polynomial identity} of \( A \) if \( f(A) = \{0\} \). On the other hand, if \( f(A) \) is contained in the center of \( A \), yet \( f \) is not identically zero on \( A \), then \( f \) is called a \textit{central polynomial} of \( A \).

With Lemmas~\ref{key} and~\ref{key1} in place, we are now fully equipped to state the following result. As its proof follows directly from the preceding lemmas, we omit the details.

\begin{theorem}\label{key2}
Let \( D \) be a division ring with infinite center \( F \), and let \( n \geq 1 \) be an integer. Assume that one of the following conditions holds:
\begin{enumerate}[\rm (i)]
	\item \( D \) is an algebraic division ring of characteristic zero;
	\item \( D \) is finite-dimensional over \( F \), where \( \operatorname{char}(F) = p > 0 \), and \( p \) does not divide \( \sqrt{\dim_FD} \) nor $n$.
\end{enumerate}
If \( f \in F\langle \mathcal{X} \rangle \) is a noncommutative polynomial that is neither a polynomial identity nor a central polynomial of \( \mathrm{M}_n(D) \), then \( \mathrm{M}_n(D) \) admits the decomposition:
\[
\mathrm{M}_n(D) = Z(\mathrm{M}_n(D)) + \mathrm{span}f(\mathrm{M}_n(D)).
\]
\end{theorem}

To close this section,  we present the following result, showing that certain division algebras can be generated by the images of a noncommutative polynomial. 

\begin{theorem}\label{generated}
	Let $D$ be a division $F$-algebra with finite dimension $n^2$ where $n\geq1$ is an integer, and let \( f \in F\langle \mathcal{X} \rangle \). Suppose there exists a subset $A\subseteq\mathrm{M}_n(\overline{F})$, where $\overline{F}$ is the algebraic closure of $F$, such that the image $f(A)$ contains an algebraic element of degree \( n \) over \( F \). Then, $D$ is generated by $f(D)$ as an $F$-algebra.
\end{theorem}

The proof of Theorem~\ref{generated} proceeds in two main steps. In the first step, we demonstrate that \( f(D) \) generates a maximal subfield of \( D \). The second step applies a result due to I.~N.~Herstein and A.~Ramer, namely \cite[Corollary 2]{Pa_He_72}, to complete the argument.

We begin the first step by recalling that a \emph{maximal subfield} of a division ring \( D \) is a subfield that is not properly contained in any larger subfield of \( D \). Such subfields play a fundamental role in understanding the structure of division rings, particularly in the finite-dimensional case, where their dimensions often align with that of the ambient ring. By Zorn’s Lemma, every division ring contains at least one maximal subfield.

To analyze maximal subfields in our setting, we follow an approach inspired by \cite{Pa_Aa_18}, which relies on rational identities in central simple algebras. For our purposes, however, we focus on the more specific case of division algebras rather than considering the full generality of the central simple case. A central example that will play a key role in our analysis is the following. Let \( n \) be a positive integer and let \( y_0, y_1, \dots, y_n \) be noncommuting indeterminates. Define the polynomial  
\[
g_n(y_0, y_1, \dots, y_n) = \sum_{\delta \in S_{n+1}} \operatorname{sign}(\delta) \, y_0^{\delta(0)} y_1 y_0^{\delta(1)} y_2 y_0^{\delta(2)}  \dots y_n y_0^{\delta(n)},
\]
where \( S_{n+1} \) denotes the symmetric group on \( \{0, 1, \dots, n\} \), and \( \operatorname{sign}(\delta) \) is the sign of the permutation \( \delta \). Introduced in \cite{Bo_Be_96}, this polynomial captures conditions for algebraicity using identities in noncommutative variables.

We now present the following lemma, which follows directly from \cite[Corollary 2.3.8]{Bo_Be_96}.

\begin{lemma} \label{algebraic}
	Let \( D \) be a finite-dimensional division algebra over a field \( F \), and let \( n \geq 1 \) be an integer. For any \( a \in \mathrm{M}_n(D) \), the following are equivalent:
	\begin{enumerate}[\rm (i)]
		\item The element \( a \) is algebraic over \( F \) of degree at most \( n \).
		\item The identity
		$g_n(a, r_1, \dots, r_n) = 0$
		holds for all \( r_1, \dots, r_n \in \mathrm{M}_n(D) \).
	\end{enumerate}
\end{lemma}

Let us recall that an element \( a \) in a ring \( R \) with center \( Z = Z(R) \) is said to be \emph{algebraic of degree \( n \)} over \( Z \) if there exists a polynomial \( f(x) \in Z[x] \) of degree \( n \) such that \( f(a) = 0 \), and no polynomial of smaller degree annihilates \( a \). Note that \( f(x) \) is not required to be irreducible, even if \( Z \) is a field.

Let \( \mathcal{I}(A) \) denote the set of all polynomial identities satisfied by a central simple algebra \( A \). The following result highlights a structural connection between such algebras and their identities, and is a direct consequence of \cite[Theorem 11]{Pa_Am_66}. 

\begin{lemma} \label{identity}
	Let \( D \) be a finite-dimensional division algebra over a field \( F \), and set \( n = \sqrt{\dim_FD} \). For any field extension \( L \) of \( F \), the following equality of polynomial identity sets holds:
	$\mathcal{I}(R) = \mathcal{I}(\mathrm{M}_n(F)) = \mathcal{I}(\mathrm{M}_n(L)).$
\end{lemma}

The following lemma confirms that the action of $g_n$ on a nonzero polynomial yields a nonzero polynomial.

\begin{lemma}\label{l2.3}
Let \( D \) be a division ring with center \( F \), and let \( f(x_1, \ldots, x_m) \in F\langle \mathcal{X} \rangle \) be a nonzero polynomial in \( m \) noncommuting variables. Then, for any positive integer \( n \), the polynomial  
$g_n(f(x_1, \ldots, x_m), y_1, \ldots, y_n)$  
is nonzero in the variables \( x_1, \ldots, x_m, y_1, \ldots, y_n \).
\end{lemma}

\begin{proof}
	By a consequence of \cite[(14.21)]{Bo_La_91}, there exists a division ring \( D_1 \) with center \( F \), containing \( D \) as a subring, such that \( \dim_F D_1 = \infty \). Assume, for contradiction, that for all elements \( a_1, a_2, \dots, a_m \in D_1 \), the evaluation \( f(a_1, a_2, \dots, a_m) \) is algebraic over \( F \) of degree at most \( n \). Then, by \cite[Theorem 1.2]{Pa_Hai.Dung.Bien_2022}, this would imply that \( D_1 \) is finite-dimensional over \( F \), which contradicts our construction of \( D_1 \). Therefore, there must exist elements \( a_1, \dots, a_m \in D_1 \) such that \( f(a_1, \dots, a_m) \) is either not algebraic over \( F \) or is algebraic of degree strictly greater than \( n \). Applying Lemma~\ref{algebraic}, we conclude that
	$g_n(f(a_1, \dots, a_m), r_1, \dots, r_n) \neq 0$
	for some \( r_1, \dots, r_n \in D_1 \). Hence, the polynomial \( g_n(f(y_1, \dots, y_m), x_1, \dots, x_n) \) is nonzero, as claimed.
\end{proof}

To investigate the existence of maximal subfields in division rings, we make use of the following lemma, which follows directly from \cite[Corollary 15.6 and Proposition 15.7]{Bo_La_91}.

\begin{lemma}  \label{maximal}
	Let \( D \) be a finite-dimensional division \( F \)-algebra, and set \( n = \sqrt{\dim_F D} \). If \( K \) is a subfield of \( D \) containing \( F \), then \( \dim_F K \leq n \), with equality if and only if \( K \) is a maximal subfield of \( D \).
\end{lemma}

With the necessary foundation in place, we now arrive at the following result, which guarantees the existence of a maximal subfield, and thus completing the initial stage of our argument.

\begin{theorem}\label{maxima}
Let $D$ be a finite-dimensional division algebra over a field $F$, and let $f(x_1, \ldots, x_m) \in F\langle \mathcal{X} \rangle$ be a nonzero noncommutative polynomial in $m$ variables. Suppose that there exists a subset $A \subseteq \mathrm{M}_n(\overline{F})$, where $\overline{F}$ is the algebraic closure of $F$, such that the image $f(A)$ contains an element algebraic of degree $n$ over $F$. Then, there exist elements $a_1, \ldots, a_m \in D$ such that the intersection of all subdivision rings of $D$ containing both $f(a_1, \ldots, a_m)$ and $F$ forms a maximal subfield of $D$.
\end{theorem}

\begin{proof}
If $F$ is finite, then $D$ is also finite-dimensional over a finite field, so the result is immediate. Assume from now on that $F$ is infinite, and set $n = \sqrt{\dim_F D}$. For any $a_1, \ldots, a_m \in D$, let us denote by $F(f(a_1, \ldots, a_m))$ the intersection of all subdivision rings of $D$ that contain both $F$ and $f(a_1, \ldots, a_m)$. This intersection is a field, and by Lemma~\ref{maximal}, the goal is to find such elements for which this field has $F$-dimension at least $n$. Define
$$
\ell = \max\{ \dim_F F(f(a_1, \ldots, a_m)) \mid a_1, \ldots, a_m \in D \}.
$$ By Lemma~\ref{algebraic},  $
g_\ell(f(a_1, \ldots, a_m), r_1, \ldots, r_\ell) = 0,
$ for all $a_1, \ldots, a_m, r_1, \ldots, r_\ell \in D$ Thus, the polynomial $g_\ell(f(x_1, \ldots, x_m), y_1, \ldots, y_n)$ vanishes identically on $D$, making it a polynomial identity of $D$ in the variables $x_1, \ldots, x_m, y_1, \ldots, y_n$. By Lemma~\ref{l2.3}, this polynomial is a nonzero polynomial. Now, Lemma~\ref{identity} tells us that any polynomial identity satisfied by $D$ must also hold in $\mathrm{M}_n(\overline{F})$. Hence,
$
g_\ell(f(A_1, \ldots, A_m), B_1, \ldots, B_\ell) = 0
$ for all $A_1, \ldots, A_m, B_1, \ldots, B_\ell \in \mathrm{M}_n(\overline{F})$. Applying Lemma~\ref{algebraic} again, this implies that all elements of the form $f(A_1, \ldots, A_m)$ are algebraic over $F$ of degree at most $\ell$. But by assumption, there exist $A_1, \ldots, A_m \in \mathrm{M}_n(\overline{F})$ such that $f(A_1, \ldots, A_m)$ is algebraic of degree $n$ over $F$. Therefore, $\ell \geq n$, as required. The proof is complete.
\end{proof}

We now make use of a result originally due to I.~N.~Herstein and A.~Ramer, stated here for convenience as it plays a crucial role in what follows (see \cite[Corollary~2]{Pa_He_72}).

\begin{lemma}\label{ma}
	Let \( D \) be a finite-dimensional division algebra over a field \( F \). If \( K \subseteq D \) is a maximal subfield generated by some element \( \alpha \in D \), then \( D \) is generated as an \( F \)-algebra by \( \alpha \) and a conjugate of \( \alpha \).
\end{lemma}

With Theorem~\ref{maxima} and Lemma~\ref{ma} in place, we are ready to complete the proof of Theorem~\ref{generated}.

\begin{proof}[The proof of Theorem~\ref{generated}]
	Suppose  \( f=f(x_1, \ldots, x_m) \in F\langle \mathcal{X} \rangle \) is a nonzero noncommutative polynomial in \( m \) variables. By Theorem~\ref{maxima}, there exist elements \( a_1, \ldots, a_m \in D \) such that the subfield \( F(f(a_1, \ldots, a_m)) \), defined as  the intersection of all subdivision rings of \(D\) that contain both \(F\) and the element \(f(a_1,\ldots,a_m)\), is a maximal subfield of \( D \). Applying Lemma~\ref{ma}, since this maximal subfield is generated by the single element \( f(a_1, \ldots, a_m) \), it follows that \( D \) is generated as an \( F \)-algebra by this element and one of its conjugates. Moreover, such a conjugate has the form
	$\gamma f(a_1, \ldots, a_m) \gamma^{-1}$
	for some nonzero element \( \gamma \in D \). Notably, this conjugate can also be written as
	$f(\gamma a_1 \gamma^{-1}, \ldots, \gamma a_m \gamma^{-1}),$
	which is again an element in the image of \( f \) evaluated on \( D \). In other words, both \( f(a_1, \ldots, a_m) \) and its conjugate belong to \( p(D) \), and together they generate \( D \) as an \( F \)-algebra. This completes the proof.
\end{proof}

	\section*{Declarations}
	
	Our statements here are the following:
	
	\begin{itemize}
		\item {\bf Ethical Declarations and Approval:} The authors have no any competing interest to declare that are relevant to the content of this article.
		\item {\bf Competing Interests:} The authors declare no any conflict of interest.
		\item  {\bf Authors' Contributions:} All three listed authors worked and contributed to the paper equally. The final editing was done by the corresponding author Tran Nam Son and was approved by all of the present authors.
		\item {\bf Availability of Data and Materials:} Data sharing not applicable to this article as no data-sets or any other materials were generated or analyzed during the current study.
	\end{itemize}
	
	\section*{Acknowledgment}
	
This research is funded by University of Transport and Communications (UTC) under grant number T2025-PHII\_KHCB-005.

	\bibliographystyle{amsplain}

\begin{thebibliography}{100}

\bibitem{Pa_Aa_18} M. Aaghabali, S. Akbari, M. H. Bien, Division algebras with left algebraic commutators, \textit{Algebr. Represent. Theory} \textbf{21} (4) (2018), 807-816.	
		
\bibitem{Pa_Al_57} A. A. Albert and B. Muckenhoupt, On matrices of trace zero, \textit{Michigan Math. J.} \textbf{4} (1957), 1-3.

\bibitem{Pa_Am_66} S. A. Amitsur, Rational identities and applications to algebra and geometry, \textit{J. Algebra} \textbf{3} (1966), 304–359. 
		
\bibitem{Pa_An_47} G. Ancochea, On semi-automorphisms of division algebras, \textit{Ann. Math.} \textbf{48} (1947), 147-153.
		
\bibitem{Pa_Ar_17} R. Archbold, L. Robert, and A. Tikuisis, The Dixmier property and tracial states for C*-algebras, \textit{J. Funct. Anal.} \textbf{273} (2017), 2655–2718.

\bibitem{Pa_Ar_23} R. J. Archbold, I. Gogic, and L. Robert, Local variants of the Dixmier property and weak centrality for C*-algebras, \textit{Int. Math. Res. Not. IMRN} (2023), 1483–1513.

\bibitem{Pa_As_66} S. Asano, On invariant subspaces of division algebras, \textit{Kōdai Math. Sem. Rep.} \textbf{18} (1966), 322–334.

\bibitem{Bo_Be_96} K. I. Beidar, W. S. Martindale, A. V. Mikhalev, \textit{Rings with Generalized Identities}, Monographs and Textbooks in Pure and Applied Mathematics, Marcel Dekker, Inc., 1996.

\bibitem{Pa_Bre_08} M. Brešar, E. Kissin, and V. S. Shulman, Lie ideals: from pure algebra to C*-algebras, \textit{J. Reine Angew. Math.} \textbf{623} (2008), 73–121.

\bibitem{Pa_Bre_20} M. Brešar, Commutators and images of noncommutative polynomials. \textit{Adv. Math.} \textbf{374} (2020), 107346, 21 pp.



	\bibitem{Pa_ChuLee_79} C. L. Chuang, P. H. Lee, Idempotents in simple rings, \textit{J. Algebra} \textbf{56} (1979), 510-515.
	
	\bibitem{Bo_Da_96} K. R. Davidson, \textit{C*-Algebras by Example}, Fields Institute Monographs, vol. \textbf{6}, American Mathematical Society, Providence, RI, 1996. 

\bibitem{Pa_Di_49} J. Dixmier, Les anneaux d’opérateurs de classe finie, \textit{Ann. Sci. Éc. Norm. Supér.} (3) \textbf{66} (1949), 209–261.



\bibitem{Pa_DuHaSo_24} T. H. Dung, B. X. Hai, T. N. Son, Reversibility in matrix rings and group algebras, \textit{Period. Math. Hungar.} \textbf{90} (2024), 203-216.

\bibitem{Bo_La_91} T. Y. Lam, \textit{A first course in noncommutative rings}, GTM 131, 2-nd ed., Springer, 1991.

\bibitem{Pa_Le_23} T-K Lee and J-H Lin, Values of polynomials on centrally closed prime algebras, \textit{J. Algebra Appl.} \textbf{22}, No. 11(2023), 2350246  (16 pages).


\bibitem{Pa_Me_06} Z. Mesyan, Commutator rings, \textit{Bull. Austral. Math. Soc.} \textbf{74} (2006), 279-288.

\bibitem{Pa_Ng_16} P. W. Ng and L. Robert, Sums of commutators in pure C*-algebras, \textit{Munster J. Math.} \textbf{9} (2016), 121–154.

\bibitem{Pa_Hai.Dung.Bien_2022} B. X. Hai, T. H. Dung, M. H. Bien, Almost subnormal subgroups in division rings with generalized algebraic rational identities, \textit{J. Algebra Appl.} \textbf{21} (2022), no. 4, Paper No. 2250075, 12 pp.

\bibitem{Pa_Ha_02} R. Hazrat, Wedderburn’s factorization theorem application to reduced K-theory, \textit{Proc. Am. Math. Soc.} \textbf{130} (2) (2002), 311-314.

\bibitem{Pa_Ha_19} R. Hazrat, A note on multiplicative commutators of division rings, \textit{J. Algebra Appl.}  \textbf{18}, No. 2 (2019) 1950031 (2 pages).

\bibitem{Pa_Her_53} I. N. Herstein, A note on a commutativity theorem, \textit{Kodai Math. Sem. Rep.} \textbf{5} (1953), 119-120.

\bibitem{Pa_He_72} I. N. Herstein,  A. Ramer,  A note on division algebras, \textit{Canadian J. Math.} \textbf{24} (1972), 734–736.

\bibitem{Pa_Thi_24} H. Thiel, Lie ideals in properly infinite C*-algebras, arXiv preprint arXiv:2412.16087, 2024.

\bibitem{Bo_Pa_89} D. S. Passman, \textit{Infinite Crossed Products}, vol. \textbf{135}, New York: Academic Press, 1989.

\bibitem{Pa_Pop_02} C. Pop, Finite sums of commutators, \textit{Proc. Amer. Math. Soc.} \textbf{130} (2002), 3039–3041.


\bibitem{Bo_Vo_21} J. Voight, \textit{Quaternion algebras}, Graduate Texts in Mathematics \textbf{288}, Springer, 2021.

		
	\end{thebibliography}

\end{document}